\documentclass[12pt,a4paper,twoside,reqno,tbtags]{amsart}
\usepackage{cite}
\usepackage{amsmath,dsfont}
\usepackage{amssymb}
\usepackage{amsthm}
\usepackage{graphicx}
\usepackage{float}
\usepackage[linkcolor=blue, urlcolor=blue, citecolor=blue, colorlinks, bookmarks]{hyperref}
\usepackage{amsfonts}
\usepackage{mathrsfs}
\usepackage{algorithm}
\allowdisplaybreaks

\allowdisplaybreaks

\usepackage[linkcolor=blue, urlcolor=blue, citecolor=blue, colorlinks, bookmarks]{hyperref}
\usepackage{amsfonts}
\usepackage{mathrsfs}
\usepackage[top=1.5cm, bottom=2cm, left=2.5cm, right=2.5cm]{geometry} 
\usepackage{amsfonts, amssymb, amscd, amsmath, enumerate, verbatim}
\usepackage{algpseudocode}
\newcommand{\mychi}{\raisebox{0pt}[1ex][1ex]{$\chi$}}
\newcounter{Theorem}

\newtheorem{corollary}[Theorem]{Corollary}
\newtheorem{remark}[Theorem]{Remark}
\newtheorem{definition}[Theorem]{Definition}
\newtheorem{example}[Theorem]{Example}
\newtheorem{theorem}[Theorem]{Theorem}
\newtheorem{lemma}[Theorem]{Lemma}
\numberwithin{Theorem}{section} \numberwithin{equation}{section}

\Large

\begin{document}
	
	\title []{ A class of weighted composition operators whose Range and Null spaces are complemented }
	
	
	\author[]{Anurag Kumar Patel}
	\address{Anurag Kumar Patel \\ Department of Mathematics \\ Banaras Hindu University \\ Varanasi 221005, India}
	\email{anuragrajme@gmail.com}
	
	
	
	

	\begin{abstract}
	In this paper, we prove that the null space of a weighted composition operator on $\ell_p~ (1 \leq p < \infty)$ is a complemented subspace. We also give a necessary and sufficient condition for a weighted composition operator on $\ell_p$ whose range space is of finite co-dimension. Thereafter, we characterize a class of weighted composition operators whose range space is a complemented subspace.	Lastly, we characterize weighted composition operators on $\ell_p$ which are Fredholm operators.	
	\end{abstract}
	
	\maketitle
	
	\section{Introduction}
		Composition operators appeared implicitly in the B. O. Koopman's formulations of classical mechanics in 1931\cite{Koopman}. J. V. Ryff\cite{Ryff} initially examined these operators in the realm of Banach spaces of  analytic functions. However,  the term ``Composition Operator" was coined  by  Nordgren in 1968 \cite{Nordgren}.
	The study of composition operators has been extensively pursued in at least two distinct contexts: (i) composition operators defined on the space of holomorphic functions and (ii) composition operators defined on $L^p(\mu)$ associated with a $\sigma$- finite measure space. Composition operators represent diverse class of operators which illuminate the whole operator theory by providing nice examples and counter examples. Further, several important problems in operator theory such as the invariant subspace problem can be simplified in terms of composition operators  \cite{Cowen paper}. Thus the study of composition operators contributes significantly to our understanding of various issues in operator theory. Consequently, the study of composition operators is an active area of research. For more details in above matter we suggest following references \cite{Cowen, Flytzans, Jbbar, Jabb, kumarwco, Limparte, Nordgren 2, RKS2}.

	Weighted composition operators are meaningful generalization of composition operators with many interesting properties. For example, all the isometries of $L_p(\mu)$ spaces are weighted composition operators \cite{AhmedR}. Further, weighted composition operators were among the prominent examples of Hypercyclic operators \cite{Rolewicz}.\\

	\textbf{Notation and Terminology :} Let $\mathbb{N}$ denote the set of all positive integers. For $1 \leq p < \infty$, $\ell_p$ denotes the Banach space of all $p$-summable real or complex sequences and $\ell_\infty$ denotes the Banach space of all bounded sequences of complex numbers. Let support of a function $u:X \to \mathbb{C}$ be denoted as $S(u)=\{x\in X: u(x)\neq 0\}.$ For Banach spaces $X$ and $Y$, we denote by $\mathcal{B}(X, Y)$ the set of all bounded linear operators from $X$ in to $Y$. For $T\in \mathcal{B}(X, Y),$ we denote the range and null space by $\mathcal{R}(T)$ and $N(T),$ respectively. For given a subset $A$ of a non-empty set $X$, the characteristic function of $A$ is denoted by $\mychi_A$ and is defined as $$ \mychi_{A}(x)=\begin{cases} 
		{1}, & x \in A \\
		
		0, & \text{otherwise.}
	\end{cases}
	$$ \\Also $|A|$ denotes the cardinality of the set A and $\text{dim}(V)$ denotes the dimension of any vector space $V$.\\
	
		\textbf{Weighted composition operator}:	Let $X$ be a non-empty set and $V(X)$ be a linear space of complex valued functions on $X$ under point wise addition and scalar multiplication. If $u:X \to \mathbb{C}$ and $\phi$ is a self-map on $X$ into itself such that composition $u\cdot f\circ\phi$ belongs to $V(X$) for each $f\in V(X)$, then $u$ and $\phi$ induces a linear transformation $uC_\phi$ on $V(X),$ defined as $$uC_\phi (f)=u\cdot f\circ\phi \quad \forall~ f \in V(X).$$
	The transformation $uC_\phi$ is known as weighted composition transformation. If $V(X)$ is a Banach space or Hilbert space and $uC_\phi$ is a bounded linear operator on $V(X),$ then $uC_\phi$ is called weighted composition operator. In particular, if $u$ becomes the identity map 1, the corresponding operator $uC_\phi$ on $V(X)$ becomes $C_\phi$ and is said to be a weighted composition operator.
	\begin{definition}
		Let $u:X \to \mathbb{C}$ and $\phi$ be a self-map on the set of natural numbers $\mathbb{N}$. Then $u$ and $\phi$ induces a linear transformation $uC_\phi$ on $\ell_p$ , defined by $$uC_\phi\left(\sum_{n=1}^{\infty}x_n\mychi_n\right) = \sum_{n=1}^{\infty}u(n)x_n\mychi_{\phi^{-1}(n)}$$ where $\mychi_n$ stands for the characteristic function of the set $\{n\}$. 
		Further, $$(uC_\phi)^nf=u\cdot (u\circ \phi) \cdots (u\circ \phi_{n-1})\cdot (f\circ\phi_n)$$ where $\phi_n$ denotes $n^{th}$ iterate of $\phi.$
	\end{definition}
	
	Let $(X,\mathfrak{B},\mu)$ be a $\sigma$-finite measure space, $L^p(X,\mathfrak{B},\mu)$ be the Banach space of all p-summable complex-valued measurable functions on X and $L^\infty(X,\mathfrak{B},\mu)$  denote the Banach space of all essentially bounded complex-valued measurable functions on X. Let $u\in L^\infty(\mu)$ and $\phi: X\to X $ be a measurable transformation. Now we state a necessary and sufficient conditions on $u$ and $\phi$ so that $uC_\phi$ become a weighted composition operator on $L^p(X,\mathfrak{B},\mu)$.
	\\ \\        	The following  results can be found in \cite{Takagi ,Nordgren 2}.
	\begin{theorem}\label{N}
		Necessary and sufficient condition for a complex valued measurable function $u$ and a measurable transformation $\phi$ to induce a bounded operator on $L^p(\mu)(1\le p < \infty)$ defined by $uC_\phi(f)=u \cdot f\circ\phi$ are $\mu\phi^{-1}<<\mu$ and $\frac{d\mu_{u^p}, \phi}{d\mu}$ is essentially bounded, where $$\mu_{u^p, \phi}(E)=\int_{\phi^{-1}(E)}|u|^pd\mu$$ and $\frac{d\mu_{u^p}, \phi}{d\mu}$ represents the Radon-Nikodym derivative.  
	\end{theorem}
	In case when $X=\mathbb{N}$ and $\mu$ is the counting measure on X, then $L^p(\mathbb{N},\mu)=\ell^p.$ In this setting theorem \ref{N} gets the following form. 
	
	\begin{theorem}\rm {\cite{RKS2}} \label{p4pre01}
		Let $u:\mathbb{N} \to \mathbb{C}$ and $\phi:\mathbb{N} \to \mathbb{N}.$ Then the weighted composition operator $uC_\phi$ on $\ell_p$ is bounded if and only if there exist a positive constant $K>0$ such that $$\sup_{n\geq1} \left(\sum_{m\in \phi^{-1}(n)}|u(m)|^p\right)\leq K.$$
	In this case $$\|uC_\phi\|=\sup_{n\geq1} \left(\sum_{m\in \phi^{-1}(n)}|u(m)|^p\right).$$
	\end{theorem}
	The following remark follows directly from the above theorem.
	\begin{remark}
		Let $u:\mathbb{N} \to \mathbb{C}$ and $\phi:\mathbb{N} \to \mathbb{N}$ be such that $uC_\phi \in  \mathcal{B}(\ell_p).$ Then $u\in \ell_\infty.$
	\end{remark}
	
%

In this paper, we study complemented subspace problem for null space and range space of a weighted composition operator on $\ell_p.$ To that end, we give a brief description of the complemented subspace problem as follows.
\begin{definition}\rm \cite{Murray}
	Let $X$ be a Banach space. Then a closed subspace $E$ of $X$ is called complemented if there is a closed subspace $F\subseteq X$(called a complementary subspace of $E$) such that $E\oplus F=X.$
\end{definition}

\textbf{Complemented subspace problem.}
The problem concerning complemented subspaces lie at the core of Banach space theory, representing a foundational aspect with a history spanning over fifty years. These concepts are pivotal in shaping and advancing the overall theory of Banach spaces. In 1937, Murray \cite{Murray} proved, for the first time, that $\ell_p,~ p\neq 2,~ p>1$ has non-complemented subspace. This significant
fact has been generalized by many mathematicians cf. \cite{Phillips,  Rosenthal}. In 1960, Pelczynski \cite{Pelczynski} showed that complemented subspaces of $\ell_1$ are isomorphic to $\ell_1$. Kothe \cite{Kothe} generalized this result to the non-separable case. In 1967, Lindenstrauss \cite{Lindenstrauss67} proved that every infinite dimensional complemented subspace of $\ell_\infty$ is isomorphic to $\ell_\infty$. This also holds if $\ell_\infty$ is replaced by $\ell_p,~ 1\leq p<\infty,~ c_0$ or $c.$ 
In 1971, Lindenstrauss and Tzafriri \cite{Lindenzafriri} proved that every infinite dimensional Banach space which is not isomorphic to a Hilbert space contains a closed non-complemented subspace. In 1993, Gowers and Maurey \cite{GowersBernard} showed that there exists a Banach space $X$ without non-trivial complemented subspaces.

We now state following well known theorems, which are used in our subsequent results.
\begin{theorem}\rm \label{p4pre03} \cite{W Rudin}
	Let $X$ be a Banach space. Then every finite dimensional subspace $E$ of $X$ is complemented.
\end{theorem}
\begin{theorem} \rm \label{p4pre04} \cite{W Rudin}
	Let $X$ be a Banach space and $E$ be its closed subspace. If $$\text{codim}(E)=\text{dim}(X \textfractionsolidus E)<\infty,$$ then $E$ is complemented.
\end{theorem}

The paper is arranged in the following manner.\\
In the second section, we find the dimension of null space of a weighted composition operator $uC_\phi$ on $\ell_p.$  We also show that the null space of $uC_\phi$ on $\ell_p$ is always complemented. Further, we also give a necessary and sufficient condition for the null space of a weighted composition operator on $\ell_p$ to be of infinite co-dimension.

In the third section, we discuss the above aspects for the range space of weighted composition operators. In the last section we characterize weighted composition operators on $\ell_p$ which are Fredholm operators.

	\begin{center}
		\textbf{Main results}
	\end{center}
	
	\section{\textbf{Weighted composition operators on $\ell_p~(1\leq p < \infty)$ spaces with finite dimensional null space} }\label{sec4.2}

	\begin{theorem}\label{npt1}
		Let $u:\mathbb{N} \to \mathbb{C}$ and $\phi:\mathbb{N} \to \mathbb{N}$ be such that $uC_\phi \in \mathcal{B}(\ell_p).$ Suppose $S(u)$ is invariant under $\phi.$ Then, $$N((uC_\phi)^n)=\{f \in \ell_p:f|_{\phi_n(S(u))}=0\}.$$
	\end{theorem}
	\begin{proof}
		Suppose that $f\in N((uC_\phi)^n).$ Then $(uC_\phi)^n(f)=0.$ Therefore, $$u\cdot (u\circ \phi) \cdots (u\circ \phi_{n-1})\cdot (f\circ \phi_n)=0.$$
		Let $g=u\cdot (u\circ \phi) \cdots (u\circ \phi_{n-1})\cdot (f\circ \phi_n).$ Now, since $S(u)$ is invariant under $\phi,$ then for each $m\in S(u)$ $$g(m)=0 \text{ implies that } f(\phi_n(m))=0.$$ 
		Therefore $$f|_{\phi_n(S(u))}=0.$$
		
		Conversely, suppose $f|_{\phi_n(S(u))}=0.$ Then, $f(\phi_n(m))=0$ for each $m\in S(u).$\\
		Hence $g(m)=0$ for all $m\in \mathbb{N}.$ This implies that $(uC_\phi)^n(f)=0.$ Therefore $$f\in N((uC_\phi)^n).$$
	\end{proof}
	The following remark is a particular case of the above Theorem
	\begin{remark}\label{npr1}
		If $u:\mathbb{N} \to \mathbb{C}$ is bounded away from zero and $uC_\phi \in \mathcal{B}(\ell_p).$ Then $$N((uC_\phi)^n)=\{f \in \ell_p:f|_{\phi_n(\mathbb{N})}=0\}.$$ 
	\end{remark}
	
	\begin{theorem}\label{npt2}
		Let $u:\mathbb{N} \to \mathbb{C}$ and $uC_\phi \in \mathcal{B}(\ell_p).$ Suppose that $S(u)$ is invariant under $\phi.$ Then, for each $n\geq 1,$ $$\text{dim}\left(N((uC_\phi)^n)\right)=|\mathbb{N}-\phi_n(S(u))|.$$
	\end{theorem}
	\begin{proof}
		Let $n\geq 1$ and $A=\mathbb{N}-\phi_n(S(u)).$ If $A$ is empty. Then, $\phi_n(S(u))=\mathbb{N}$ and hence, by Theorem \ref{npt1}, $N((uC_\phi)^n)=\{0\}.$ Therefore, $$\text{dim}(N((uC_\phi)^n=0.$$
		
		Now, suppose $A$ is a non-empty set.  Since for each $n\in A,$ $$\mychi_{n}|_{\phi_n(S(u))}=0.$$
		Hence  by Theorem \ref{npt1}, $\mychi_{n}\in N((uC_\phi)^n) \text{ for each } n\in A.$

		We claim that $\mathfrak{B}=\{\mychi_{n}:n\in A\}$ is a basis for $N((uC_\phi)^n).$ We first prove that $\mathfrak{B}$ is linearly independent. Let $\{n_1, n_2,...n_k\}\subset A$ and $\lambda_{1}, \lambda_{2},..., \lambda_{k}\in \mathbb{F}$ be such that $\sum_{i=1}^{k}\lambda_i \mychi_{n_i}=0.$
		Let $g=\sum_{i=1}^{k}\lambda_i \mychi_{n_i}.$ Then, $$g(n_j)=\lambda_{j}=0 \text{ for each } 1\leq j\leq k.$$ Hence, $\mathfrak{B}$ is linearly independent.
		
		Now, we prove that $\mathfrak{B}$ is a spanning set for $N((uC_\phi)^n).$ Let $f=\sum_{n=1}^{\infty}f(n)\mychi_n \in N((uC_\phi)^n).$ Then $f=\sum\limits_{n\in A}f(n)\mychi_n$ since $f|_{\phi_n(S(u))}=0.$ Therefore $\mathfrak{B}$ is a spanning set for $N((uC_\phi)^n).$	
		\end{proof}
	
	\begin{corollary}\label{npc1}
		Let $u:\mathbb{N} \to \mathbb{C}$ and $uC_\phi \in \mathcal{B}(\ell_p).$ Suppose that $S(u)$ is invariant under $\phi.$ Then dimension of $N((uC_\phi)^n)$ is finite if and only if~ $\mathbb{N}-\phi_n(S(u))$ is a finite set. In this case, $$\text{dim}(N((uC_\phi)^n))=|\mathbb{N}-\phi_n(S(u))|.$$
	\end{corollary}
	\begin{proof}
		Let $A=\mathbb{N}-\phi_n(S(u)).$ Then by the Theorem \ref{npt2}, $\mathfrak{B}=\{\mychi_{n}:n\in A\}$ is a basis for $N((uC_\phi)^n)$ . Hence $\text{dim}(N((uC_\phi)^n))$ is finite if and only if $A$ is finite. This completes the proof. 
	\end{proof}	
	
	We now give two examples, one in which the dimension of the null space of $uC_\phi$ is finite and one in which the dimension of the null space of $uC_\phi$ is infinite.
	\begin{example}
		Let $u:\mathbb{N} \to \mathbb{C}$ be defined as 
		$$ u(n)=\left\{\begin{array}{rcl} 0 & \mbox{if} & n=1, 2 \\ \frac{1}{n} & \mbox{if} & n>2
		\end{array}\right.$$
	and $\phi:\mathbb{N} \to \mathbb{N}$ be defined as $\phi(n)=n~\forall~n \in \mathbb{N}.$ Clearly $S(u)=\{n:n\geq3\}.$ Hence $\phi(S(u))=\{n:n\geq4\}$ and $\mathbb{N}-\phi(S(u))=\{1, 2, 3\}.$ Therefore by Corollary \ref{npc1}, $\text{dim}(N(uC_\phi))=3.$
	\end{example}
	\begin{example}
		Let $p_n$ denote the $n^{th}$ prime.  Define $u:\mathbb{N} \to \mathbb{C}$ as $$ u(n)=\left\{\begin{array}{rcl} 0 & \mbox{if} & 1\leq n \leq 10 \\ \frac{1}{n} & \mbox{if} & n>10
				\end{array}\right.$$
		and define $\phi:\mathbb{N} \to \mathbb{N}$ as $\phi(n)=p_n.$ 
		Then $S(u)=\{n:n>10\}$ and $\phi(S(u))\subseteq S(u).$
		Hence by Corollary \ref{npc1}, $\text{dim}(N(uC_\phi))=|\mathbb{N}-\phi(S(u))|=\infty.$
	\end{example}
	
	The following lemma gives a necessary and sufficient condition on $\phi$ for null space of $C_\phi$ to be an invariant subspace of $C_\phi.$
	\begin{lemma}\label{raju1}
	Let $u:\mathbb{N} \to \mathbb{C}$ and $uC_\phi$ is a bounded linear operator on $\ell_p.$ Suppose $S(u)$ is invariant under $\phi.$ Then $N((uC_\phi)^2)=N(uC_\phi)$ if and only if $\phi_2(S(u))=\phi(S(u))$.	
	\end{lemma}
	
	\begin{proof}
	Suppose  $N((uC_\phi)^2)=N(uC_\phi).$ Then it follows from \ref{npt1} that $$\phi_2(S(u))=\phi(S(u)).$$
	
	Conversely, if $\phi_2(S(u))=\phi(S(u)).$ Then 
	\begin{align*}
		N((uC_\phi)^2)&=\{f\in \ell_p:f|_{\phi_2(S(u))=0}\}\\&=\{f\in \ell_p:f|_{\phi(S(u))=0}\}\\&=N(uC_\phi).
	\end{align*}	
	\end{proof}
	
	
	

	\begin{lemma}\label{npl1}
		Suppose $u:\mathbb{N} \to \mathbb{C}$ is bounded away from zero and $uC_\phi \in \mathcal{B}(\ell_p).$ Then $\text{dim}(N(uC_\phi))<\infty$ if and only if for each $n\geq 1$ $$\text{dim}(N((uC_\phi)^n))<\infty.$$
	\end{lemma}
	
	\begin{proof}
		Suppose $\text{dim}(N(uC_\phi))<\infty.$ Then by Corollary \ref{npc1}, $\mathbb{N}- \phi(S(u))$ is a finite set. Hence this together with the fact that $\phi_k(\mathbb{N})- \phi_{k+1}(S(u))\subseteq \phi_k\left(\mathbb{N}- \phi(S(u))\right)$ implies that $\phi_k(\mathbb{N})- \phi_{k+1}(S(u))$ is finite set for each $k\geq1$. Now $$\mathbb{N}-\phi_n(S(u))=\bigcup_{k=1}^{n}\left(\phi_{k-1}(\mathbb{N})- \phi_k(S(u))\right)$$ is a finite union of finite sets. Hence $\mathbb{N}- \phi_n(S(u))$ is a finite set. Therefore by Remark \ref{npr1}, $$\text{dim}(N((uC_\phi)^n))<\infty.$$
		Conversely, suppose that $\text{dim}(N((uC_\phi)^n)))<\infty$ for each $n\geq 1.$ Then $$\text{dim}(N(C_\phi))<\infty.$$ 
	\end{proof}

%
	
	The following corollary is the negation of the Lemma \ref{npl1} 
	\begin{corollary}\label{nullc1.04}
		Let $u:\mathbb{N} \to \mathbb{C}$ be such that $uC_\phi \in \mathcal{B}(\ell_p).$ Then $\text{dim}(N(uC_\phi))=\infty$ if and only if $\text{dim}(N((uC_\phi)^n)))=\infty$ for all $n\in \mathbb{N}.$
	\end{corollary}
	
		The following theorem shows that the null space of a weighted composition operator is always a complemented subspace.
	\begin{theorem}
		Suppose $u:\mathbb{N} \to \mathbb{C}$ and $uC_\phi \in \mathcal{B}(\ell_p).$ Then $N(uC_\phi)$ is complemented.
	\end{theorem}
	\begin{proof}
		Let $uC_\phi \in \mathcal{B}(\ell_p).$ Then $$N(uC_\phi)=\{f\in \ell_p:f|_{\phi(S(u))}=0\}.$$
		Let $M=\{f\in \ell_p:f|_{\mathbb{N}-\phi(S(u))}=0\}.$ Clearly, $M$ is a closed set. \\\
		Let $f\in N(uC_\phi) \cap M.$ Then $$f|_{\phi(S(u))}=0 \text{ and } f|_{\mathbb{N}-\phi(S(u))}=0.$$
		This implies that $f=0.$ Hence $N(uC_\phi) \cap M=\{0\}.$
		
		Now let $f=\sum \limits_{n\in \mathbb{N}}f_n\mychi_n\in \ell_p.$ Then $$f=\sum \limits_{n\in (\mathbb{N}-\phi(S(u)))}f_n\mychi_n+\sum \limits_{n\in \phi(S(u))}f_n\mychi_n=g+h,$$ where $g\in N(uC_\phi)$ and $h\in M.$
		Hence $f\in N(uC_\phi)+M.$ Therefore $N(uC_\phi)$ is complemented.
		
	\end{proof}
	
	In the following theorem we find the codimension of $N(uC_\phi)$ in $\ell_p.$
	
	\begin{theorem}\label{nullt2.00}
		Let $u:\mathbb{N} \to \mathbb{C}$ be such that $uC_\phi \in \mathcal{B}(\ell_p).$ Then $$\text{dim}(\ell_p \textfractionsolidus N(uC_\phi))=\left|\phi(S(u))\right|.$$
	\end{theorem}
	\begin{proof}
		 Observe that for each $n\in \phi(S(u)),$  
		 $$\phi^{-1}(n) \cap S(u) \neq \emptyset.$$
		 
		Hence by Theorem \ref{npt1} $$\mychi_n \notin N(uC_\phi)) \quad \text{ for each } n \in \phi(S(u)).$$
		We claim that $\mathfrak{B}=\{\mychi_{n}+N(uC_\phi):n \in \phi(S(u))\}$ is a basis for $\ell_p \textfractionsolidus N(uC_\phi).$
		Let $\lambda_{1}, \lambda_{2},..., \lambda_{k}\in \mathbb{F}$ and $n_1, n_2,...n_k \in \phi(S(u))$ be such that
		$$\sum_{i=1}^{k}\lambda_i(\mychi_{n_i}+N(uC_\phi))=N(uC_\phi).$$
		Then $\sum_{i=1}^{k}\lambda_i\mychi_{n_i} \in N(uC_\phi).$
		
		Let $g=\sum_{i=1}^{k}\lambda_i\mychi_{n_i}.$ Since $g\in N(uC_\phi),$ then $g|_{\phi(S(u))}=0.$ Therefore for each $1 \leq j \leq k,$ $$g(n_j)=\lambda_j=0.$$
		This implies that $\mathfrak{B}$ is linearly independent.

		Now let $f+N(uC_\phi)=\sum\limits_{n\in \mathbb{N}}f(n)\mychi_n+N(uC_\phi) \in \ell_p \textfractionsolidus N(uC_\phi).$ Then \begin{align*}
			f+N(uC_\phi)&=\sum\limits_{n\in \phi(S(u))}f(n)\mychi_n+\sum\limits_{n\notin \phi(S(u))}f(n)\mychi_n+ N(uC_\phi)\\&=\sum\limits_{n\in \phi(S(u))}f(n)\mychi_n+N(uC_\phi).
		\end{align*}
		Hence $\mathfrak{B}$ is a spanning set for $\ell_p \textfractionsolidus N(uC_\phi).$ Therefore $$\text{dim}(\ell_p \textfractionsolidus N(uC_\phi))=\left|\phi(S(u))\right|.$$
	\end{proof}

	The following lemma is the direct consequence of the above Theorem \ref{nullt2.00}. In this lemma, we show that codimension of $N(uC_\phi)$ in $\ell_p$ is infinite if and only if $ \phi(S(u))$ is an infinite set. Hence the converse of the Theorem \ref{p4pre04} is not true.
	\begin{lemma}\label{nullc1.05}
		Suppose $u:\mathbb{N} \to \mathbb{C}$ and $uC_\phi \in \mathcal{B}(\ell_p).$ Then $\ell_p \textfractionsolidus N(uC_\phi)$ is infinite dimensional if and only if $\phi(S(u))$ is an infinite set.
	\end{lemma}
	
	Now we give below two examples- one in which co-dimension of $N(uC_\phi)$ is infinite and another- in which  $N(uC_\phi)$ has finite co-dimension. 
	
	\begin{example}
		 Define $u:\mathbb{N} \to \mathbb{C}$ as $$ u(n)=\left\{\begin{array}{rcl} 0 & \mbox{if} & n  \text{ is odd } \\ \frac{1}{n^3} & \mbox{if} & n \text{ is even } 
	\end{array}\right.$$
	and define $\phi:\mathbb{N} \to \mathbb{N}$ as $\phi(n)=n^2.$ 
	Then, we have $S(u)=\{n:n \text{ is even natural number }\}$ and $\phi(S(u))$ is an infinite set.
	Hence by Lemma \ref{nullc1.05},$$\text{dim}(\ell_p \textfractionsolidus N(uC_\phi))=\infty.$$	
	\end{example}
	
	\begin{example}
		Define $u:\mathbb{N} \to \mathbb{C}$ as $u(n)=\frac{1}{n^2}$ and $\phi:\mathbb{N} \to \mathbb{N}$ as $\phi(n)=1.$ Then, we have $S(u)=\mathbb{N}$ and $\phi(S(u))=\{1\}.$ Hence by Theorem \ref{nullt2.00}, $$\text{dim}(\ell_p \textfractionsolidus N(uC_\phi))=1.$$
	\end{example}
	
\section{ \textbf{ Weighted composition operator on $\ell_p~(1\leq p<\infty)$ spaces whose range space is of finite co-dimension } } \label{sec4.3}
In this section, we characterize weighted composition operators on $\ell_p$ space whose range space has finite co-dimension. Thereafter, we use the forgoing characterization to give a class of weighted composition operators whose range space is a complemented subspace. \\
We begin with the following lemma, which shows that the range space of a weighted composition operator $uC_\phi$ on $\ell_p$ is always infinite dimensional. 

\begin{theorem}
Suppose $u:\mathbb{N} \to \mathbb{C}$ and $uC_\phi \in \mathcal{B}(\ell_p).$ Then dimension of $\mathcal{R}(uC_\phi)$ is infinite if $\{n:\phi^{-1}(n)\cap S(u) \neq \emptyset\}$ is an infinite subset of $\mathbb{N}.$	
\end{theorem}
\begin{proof}
	Let $A=\{n:\phi^{-1}(n)\cap S(u) \neq \emptyset\}.$
 Observe that for each $n\in A$ $$u\mychi_{\phi^{-1}(n)}=uC_\phi(\mychi_n)\in \mathcal{R}(uC_\phi).$$
 We claim that $\mathfrak{B}=\{u\mychi_{\phi^{-1}(n)}: n\in A\}$ is linearly independent. 
 Let $\lambda_{1}, \lambda_{2},..., \lambda_{k}\in \mathbb{F}$ and $n_1, n_2,...n_k \in A$ be such that $$\lambda_{1}u\mychi_{\phi^{-1}(n_1)}+\lambda_{2}u\mychi_{\phi^{-1}(n_2)}+...+\lambda_{k}u\mychi_{\phi^{-1}(n_k)}=0.$$
 Let $g=\lambda_{1}u\mychi_{\phi^{-1}(n_1)}+\lambda_{2}u\mychi_{\phi^{-1}(n_2)}+...+\lambda_{k}u\mychi_{\phi^{-1}(n_k)}.$ Since $\phi^{-1}(n_i)\cap S(u) \neq \emptyset$ for each $1\leq i \leq k.$ Therefore for $m\in \phi^{-1}(n_i)\cap S(u),$ \begin{align*}
 	g(m)&=\lambda_{i}u(m)\mychi_{\phi^{-1}(n_i)}(m)\\&=\lambda_{i}u(m)
 \end{align*}
 Since $g=0$ and $m\in S(u),$ therefore $$g(m)=0 ~\text{ implies }~ \lambda_{i}=0~ \text{ for each } 1\leq i \leq k.$$
 
Hence $\mathfrak{B}$ is a linearly independent subset of $\mathcal{R}(uC_\phi).$ 
\end{proof}

The following corollary is the direct consequence of the above theorem
\begin{corollary}
	 Suppose $u:\mathbb{N} \to \mathbb{C}$ is bounded away from zero and $uC_\phi \in \mathcal{B}(\ell_p).$ Then $\mathcal{R}(uC_\phi)$ is infinite dimensional.
\end{corollary}

\begin{theorem}
	Let $u:\mathbb{N} \to \mathbb{C}$ and $\phi:\mathbb{N} \to \mathbb{N}$ be such that $S(u)$ is invariant under $\phi$ and both $C_\phi$ and $uC_\phi$ are bounded linear operators on $\ell_p.$ Suppose $M_k={\phi^{-1}(k)}\bigcap S(u)$ for each $k\in \mathbb{N}.$ Then range space $\mathcal{R}(uC_\phi)$ of $uC_\phi$ is given by
	\begin{multline*} \mathcal{R}(uC_\phi)=\left\{f\in \ell_p:\sum_{k\in S(u)}|M_k|\left|\left(\frac{f}{u}\right)(k)\right|^p<\infty, f(k)=0 \text{ whenever } k\in \mathbb{N}-S(u) \right.\\ \left.\text{ and } \right.\left.\left.\left(\frac{f}{u}\right)\right|_{M_k}=\text{ constant for each } k\in \mathbb{N}\right\}
	\end{multline*}
\end{theorem}
\begin{proof}
	Let $f \in \mathcal{R}(uC_\phi).$ Then there exists $g \in \ell_p$ such that
	
	$$uC_\phi(g)=f.$$ \\ 
	This implies that $$f(k)=0 \text{ whenever } k\in \mathbb{N}-S(u) \text{ and } \left(\frac{f}{u}\right)(k)=(g\circ\phi)(k) \text{ for each } k\in S(u).$$\\
	Now, as $\{|M_k|:k\in \mathbb{N}\}$ is a bounded subset of $\mathbb{N}$ and $g\circ \phi\in \ell_p,$ therefore it follows that $$\sum_{k\in S(u)}|M_k|\left|\left(\frac{f}{u}\right)(k)\right|^p<\infty.$$	
	
	Further, let $m_1, m_2\in M_k.$ Then 
	$(g\circ \phi)(m_1)=(g\circ \phi)(m_2).$ Hence $$\left(\frac{f}{u}\right)(m_1)=\left(\frac{f}{u}\right)(m_2).$$
	
	This implies that $$\left(\frac{f}{u}\right)\Big|_{M_k}=\text{ constant }.$$
	
	Conversely, suppose $f\in \ell_p,~\sum\limits_{k\in S(u)}|M_k|\left|\left(\frac{f}{u}\right)(k)\right|^p<\infty,$ $$f(k)=0 \text{ whenever } k\in \mathbb{N}-S(u) \text{ and } \left(\frac{f}{u}\right)\Big|_{M_k}=\text{ constant for each }k\in \mathbb{N}.$$\\
	Define $g:\mathbb{N} \to \mathbb{C}$ as follow 
	$$g(k)=\begin{cases}
		\left(\frac{f}{u}\right)(m), & \text{ for some }m\in M_k \text{ if } M_k \text{ is non-empty }\\
		0, & \text{ if } M_k \text{ is empty. }
	\end{cases}$$
	Then $g$ is well defined as $\left(\frac{f}{u}\right)$ is constant on $M_k$ for each $k\in \mathbb{N}.$ From the above conditions on $f,$ it can be easily seen that $g$ belongs to $\ell_p.$\\
	Further, 
	\begin{align*}
		uC_\phi(g)(k)=&u(k)g(\phi(k))\\=&f(k) \text{ for all }k\in \mathbb{N}.
	\end{align*}
	Thus $uC_\phi(g)=f.$ Hence $f\in \mathcal{R}(uC_\phi).$
\end{proof}

The following theorem characterizes the weighted composition operators on $\ell_p$ whose range has finite co-dimension.

\begin{theorem}\label{complet0.01}
	Let $u:\mathbb{N} \to \mathbb{C}$ and $\phi:\mathbb{N} \to \mathbb{N}$ be such that $C_\phi$ and $uC_\phi$ are both bounded linear operators on $\ell_p.$ Then $\mathcal{R}(uC_\phi)$ has finite co-dimension in $\ell_p$ if and only if\\ $\{n\in \mathbb{N}:|\phi^{-1}(n)\cap S(u)|>1\}$ is a non-empty finite set.
\end{theorem}
\begin{proof}
	Let $A=\{n\in \mathbb{N}:|\phi^{-1}(n)\cap S(u)|>1\}.$ 
	Suppose $A$ is a non-empty finite set. Let $A=\{n_1, n_2,...,n_k\}$ and suppose for each $i\in\{1, 2,...,k\}$  $$A_i=\phi^{-1}(n_i)\cap S(u)=\{m_{j}^i:j=1, 2,...,l_i\}$$
	Observe that $\sum\limits_{n\notin \phi^{-1}(A)}u(n)f(n)\mychi_n\in \mathcal{R}(uC_\phi).$ \\
	We claim that $\mathfrak{B}=\{\mathcal{R}(uC_\phi)+u(m_{j}^i)\mychi_{m_{j}^i}:i\in\{1, 2,...,k\} ~\text{and}~j=1, 2,...,l_i-1\}$ forms a Hamel basis for $\ell_p\textfractionsolidus \mathcal{R}(uC_\phi).$ We first prove that $\mathfrak{B}$ is linearly independent.\\
	Let $\lambda_{j}^{i}\in \mathbb{F}$ for all $1\leq i\leq k$ and $1\leq j \leq l_i-1$ be such that \\ $$\sum_{i=1}^{k}\left(\sum_{j=1}^{l_i-1}\lambda_{j}^{i}(\mathcal{R}(uC_\phi)+u(m_{j}^i)\mychi_{m_{j}^{i}})\right)=\mathcal{R}(uC_\phi).$$
	Then $$\sum_{i=1}^{k}\sum_{j=1}^{l_i-1}\lambda_{j}^{i}u(m_{j}^i)\mychi_{m_{j}^{i}}\in \mathcal{R}(uC_\phi).$$
	Let $g=\sum_{i=1}^{k}\sum_{j=1}^{l_i-1}\lambda_{j}^{i}u(m_{j}^i)\mychi_{m_{j}^{i}}.$ Then, for each $1\leq i\leq k \text{ and } 1\leq j\leq l_i-1$ $$\left(\frac{g}{u}\right)({m_{j}^{i}})=\lambda^i_j \text{ and } \left(\frac{g}{u}\right)({m^i_{l_i}})=0.$$
	Now $g\in \mathcal{R}(uC_\phi),$ therefore $\left(\frac{g}{u}\right)\Big|_{A_i}=\text{constant}$ for all $1\leq i\leq k.$
	This implies that $$\lambda^i_j=0 \text{ for all } 1\leq i\leq k \text{ and } 1\leq j \leq l_i-1.$$\\
	Now, we show $\mathfrak{B}$ generates $\ell_p\textfractionsolidus \mathcal{R}(uC_\phi).$ Let $ \mathcal{R}(uC_\phi)+f \in\ell_p\textfractionsolidus \mathcal{R}(uC_\phi) \text{ where } f=(f(1), f(2), f(3),...).$
	Then
	\begin{align*}
		\mathcal{R}(uC_\phi)+(f(1), f(2), f(3),...)&=\mathcal{R}(uC_\phi)+\sum\limits_{n\notin \phi^{-1}(A)}u(n)f(n)\mychi_n+\sum\limits_{n\in \phi^{-1}(A)}u(n)f(n)\mychi_n \\&=\mathcal{R}(uC_\phi)+\sum\limits_{n\in A_1}u(n)f(n)\mychi_n+\sum\limits_{n\in A_2}u(n)f(n)\mychi_n+...\\
		&\quad+\sum\limits_{n\in A_k}u(n)f(n)\mychi_n.
	\end{align*}
	Note that $$\sum\limits_{n\in A_i}u(n)f(m_{l_i}^i)\mychi_n\in \mathcal{R}(uC_\phi)\text{ for each }1\leq i \leq k.$$
	Therefore \begin{align*}
		\mathcal{R}(uC_\phi)+\sum\limits_{n\in \mathbb{N}}u(n)f(n)\mychi_n&=\mathcal{R}(uC_\phi)+\sum_{i=1}^{k}\left(\sum\limits_{n\in A_i}u(n)f(n)\mychi_n-\sum\limits_{n\in A_i}u(n)f(m_{l_i}^i)\mychi_n\right)\\&=\mathcal{R}(uC_\phi)+\sum_{i=1}^{k}\left(\sum\limits_{n\in A_i\setminus\{m_{l_i}^i\}}u(n)\left(f(n)-f(m_{l_i}^i)\right)\mychi_n\right)
	\end{align*}
	which is a linear combination of elements from $\mathfrak{B}.$	\\
	Therefore $\mathfrak{B}$ forms a Hamel basis for $\ell_p\textfractionsolidus \mathcal{R}(uC_\phi)$ with dimension $$\text{dim}\left.(\ell_p\textfractionsolidus \mathcal{R}(uC_\phi)\right.)=\sum\limits_{n\in A}\left(|\phi^{-1}(n_i)\cap S(u)|-1\right).$$

	Conversely, suppose the set $A$ is not finite. Then there exists a sequence $\{n_i\}_{i=1}^\infty$ of positive integers such that
	$|\phi^{-1}(n_i)\cap S(u)|>1$ for each $i\geq1.$
	Let $$A_i=\phi^{-1}(n_i)\cap S(u)=\{m_{j}^i:j=1, 2,...,l_i\} \text{ for each } i\geq1.$$\\
	We claim that $\mathfrak{B}=\left\{\mathcal{R}(uC_\phi)+u(m_1^{1})\mychi_{m_1^{1}}, \mathcal{R}(uC_\phi)+u(m_1^{2})\mychi_{m_1^{2}},...,\mathcal{R}(uC_\phi)+u(m_1^{k})\mychi_{m_1^{k}},...\right\}$ is a linearly independent set in $\ell_p\textfractionsolidus \mathcal{R}(uC_\phi).$

	Let $\lambda_{1}, \lambda_{2},..., \lambda_{k}\in \mathbb{F}$ be such that $\sum_{i=1}^{k}\lambda_i(\mathcal{R}(uC_\phi)+u(m_1^{i})\mychi_{m_1^{i}})=\mathcal{R}(uC_\phi).$ Then
	$$\sum_{i=1}^{k}\lambda_iu(m_1^{i})\mychi_{m_1^{i}} \in \mathcal{R}(uC_\phi).$$
	Let $g=\sum_{i=1}^{k}\lambda_iu(m_1^{i})\mychi_{m_1^{i}}.$ Then, for each $1\leq i\leq k $ $$\left(\frac{g}{u}\right)({m_{1}^{i}})=\lambda_i \text{ and } g\left(\frac{g}{u}\right)({m^i_{2}})=0.$$
	Now $g\in \mathcal{R}(uC_\phi),$ therefore $\left(\frac{g}{u}\right)\Big|_{A_i}=\text{constant}$ for each $i\geq 1.$
	This implies that $$\lambda_i=0 ~\text{ for all } ~1\leq i\leq k.$$
	This completes the proof.
\end{proof}

The following corollary is an immediate consequence of the above Theorem \ref{complet0.01}.
\begin{corollary}\label{{compleL0.01}}
	 Let $u:\mathbb{N} \to \mathbb{C}$ be such that $uC_\phi \in
	\mathcal{B}(\ell_p).$ Then $$\text{dim}(\ell_p\textfractionsolidus \mathcal{R}((uC_{\phi})^n))<\infty$$ if and only if $\{k\in \mathbb{N}:|\phi^{-1}_n(k)\cap S(u)|>1\}$ is a finite set.
\end{corollary}
Now we give an example which satisfies Theorem \ref{complet0.01}. 
\begin{example}
		Let $u:\mathbb{N} \to \mathbb{C}$ be defined as $$ u(n)=\left\{\begin{array}{rcl} 0 & \mbox{if} & 1\leq n \leq 5 \\ 1+\frac{1}{n} & \mbox{if} & n>5
	\end{array}\right.$$
	and $\phi:\mathbb{N} \to \mathbb{N}$ be defined as
	$$ \phi(n)=\left\{\begin{array}{rcl} 1 & \mbox{if} & 1\leq n \leq 7 \\ n+1 & \mbox{if} & n>7
	\end{array}\right.$$
	Then $\{n\in \mathbb{N}:|\phi^{-1}(n) \cap S(u)|>1\}=\{1\}.$ Hence by Theorem \ref{complet0.01}, $$\text{dim}(\ell_p\textfractionsolidus \mathcal{R}(uC_\phi))=1.$$
\end{example}

The next theorem gives a class of weighted composition operators whose range is a complemented subspace.
\begin{theorem}
Let $u:\mathbb{N} \to \mathbb{C}$ be such that $uC_\phi \in \mathcal{B}(\ell_p)$ and $\mathcal{R}(uC_\phi)$ is closed. Then $\mathcal{R}(uC_\phi)$ is complemented if $\{n\in \mathbb{N}:|\phi^{-1}(n) \cap S(u)|>1\}$ is a finite set.
\end{theorem}
\begin{proof}
Let $A=\{n\in \mathbb{N}:|\phi^{-1}(n) \cap S(u)|>1\}$ be a finite set. Then by Theorem \ref{complet0.01}, $$\text{dim}(\ell_p \textfractionsolidus \mathcal{R}(uC_\phi))<\infty.$$  Therefore by Theorem \ref{p4pre04}, $\mathcal{R}(uC_\phi)$ is complemented if $A$ is a finite set.	
\end{proof}

\begin{remark}
If $\{k\in \mathbb{N}:|\phi^{-1}_n(k) \cap S(u)|>1\}$ is a finite set. Then by above Lemma \ref{{compleL0.01}}, $$\text{dim}(\ell_p\textfractionsolidus \mathcal{R}((uC_{\phi})^n))<\infty.$$ Consequently by Theorem \ref{complet0.01},   
$\mathcal{R}((uC_{\phi})^n)$ is complemented.	
\end{remark}

\begin{corollary}\label{compleC0.01}
	Suppose $u:\mathbb{N} \to \mathbb{C}$ and $uC_\phi \in \mathcal{B}(\ell_p).$ If $\{k\in \mathbb{N}:|\phi^{-1}_n(k)| \cap S(u)>1\}$ is finite, then $$\text{dim}\left(\mathcal{R}((uC_\phi)^{n-1})\textfractionsolidus \mathcal{R}((uC_{\phi})^n)\right)<\infty.$$
\end{corollary}
\begin{proof}
	Let $\{k\in \mathbb{N}:|\phi^{-1}_n(k) \cap S(u)|>1\}$ is finite, then by Lemma \ref{{compleL0.01}}  $$\text{dim}\left(\ell_p\textfractionsolidus \mathcal{R}((uC_{\phi})^n)\right)<\infty.$$
	Since $\mathcal{R}((uC_\phi)^{n-1})\subseteq \ell_p,$ therefore $\mathcal{R}((uC_\phi)^{n-1})\textfractionsolidus \mathcal{R}((uC_{\phi})^n) \subseteq \ell_p\textfractionsolidus \mathcal{R}((uC_{\phi})^n)$ and hence  $$\text{dim}\left(\mathcal{R}((uC_\phi)^{n-1})\textfractionsolidus \mathcal{R}((uC_{\phi})^n)\right)<\infty.$$
\end{proof}
\begin{remark}
	If $\phi:\mathbb{N} \to \mathbb{N}$ and $n\in \mathbb{N}$ are such that $\{k\in \mathbb{N}:|\phi^{-1}_n(k)\cap S(u)|>1\}$ is finite. Then from the Theorem \ref{p4pre04}, $\mathcal{R}((uC_\phi)^{n-1})$ is complemented in $\mathcal{R}((uC_{\phi})^n).$
\end{remark}

\section{ \textbf{ Weighted composition operator as Fredholm operator on $\ell_p$ spaces $(1\leq p<\infty)$ } } \label{sec4.4}

In this section, using the results of previous sections, we characterize weighted composition operators which are Fredholm operators. To that end, we first define Fredholm operator.\\ 

\textbf{Fredholm operator.} \rm \cite{Salamon} Let $X$ and $Y$ be Banach spaces. An operator $T\in \mathcal{B}(X, Y)$ is called a Fredholm operator if $\mathcal{R}(T)$ is closed, $\text{dim}(N(T))$ and $\text{dim}(\mathcal{R}(T))$ are both finite. In this case, the (Fredholm) index of $T$ is defined as $$\text{ind}(T)=\text{dim}(N(T))- \text{codim}(\mathcal{R}(T)).$$

The following theorem characterizes weighted composition operators on $\ell_p$ which are Fredholm operators.
\begin{theorem}\label{complsec4.01}
	Let $u:\mathbb{N} \to \mathbb{C}$ be such that $uC_\phi \in \mathcal{B}(\ell_p)$ and $\mathcal{R}(uC_\phi)$ is closed. Then $uC_\phi$ is a Fredholm operator if and only if $|\mathbb{N}-\phi(S(u))|<\infty$ and $|A|<\infty,$ where $$A=\{n\in \mathbb{N}:|\phi^{-1}(n) \cap S(u)|>1\}.$$ In this case, Fredholm index of $uC_\phi$ is $$\text{ind}(uC_\phi)=|\mathbb{N}-\phi(S(u))|-\sum\limits_{n\in A}\left(|\phi^{-1}(n)|-1\right).$$
\end{theorem}
\begin{proof}
Suppose $uC_\phi \in \mathcal{B}(\ell_p)$ is a Fredholm operator. Then $$\text{dim}(N(uC_\phi))<\infty \text{ and } \text{codim}(\mathcal{R}(uC_\phi))<\infty.$$ Consequently by Corollary \ref{npc1}, $|\mathbb{N}-\phi(S(u))|<\infty$ and, by Theorem \ref{complet0.01}, $|A|<\infty.$

Conversely, suppose $|\mathbb{N}-\phi(S(u))|<\infty$ and $|A|<\infty.$ Then by Corollary \ref{npc1}, $$\text{dim}(N(uC_\phi))<\infty$$ and, by Theorem \ref{complet0.01} $$\text{codim}(\mathcal{R}(uC_\phi))<\infty.$$
Therefore $uC_\phi$ is a Fredholm operator. In this case
\begin{align*}
	\text{ind}(uC_\phi)&=\text{dim}(N(uC_\phi))-\text{codim}(\mathcal{R}(uC_\phi))\\&=|\mathbb{N}-\phi(S(u))|-\sum\limits_{n\in A}\left(|\phi^{-1}(n)\cap S(u)|-1\right).
\end{align*}
\end{proof}
Now we construct examples of  weighted composition operators on $\ell_p$ which are Fredholm operators with positive, zero and negative indices respectively. 
\begin{example}
	Let $u:\mathbb{N} \to \mathbb{C}$ be defined as $$ u(n)=\left\{\begin{array}{rcl} 0 & \mbox{if} & 1\leq n \leq 3 \\ 1+\frac{1}{n} & \mbox{if} & n>3
	\end{array}\right.$$
	and $\phi:\mathbb{N} \to \mathbb{N}$ be defined as
	$$ \phi(n)=\left\{\begin{array}{rcl} 1 & \mbox{if} & 1\leq n \leq 5 \\ n & \mbox{if} & n>5
	\end{array}\right.$$
	Then $S(u)=\{n:n\geq 4\},~\mathbb{N}-\phi(S(u))=\{2, 3, 4, 5\}$ and  $$A=\{n\in \mathbb{N}:|\phi^{-1}(n)\cap S(u)|>1\}=\{1\}.$$
	Hence \begin{align*}
		\sum\limits_{n\in A}\left(|\phi^{-1}(n)\cap S(u)|-1\right)&=|\phi^{-1}(1)\cap \{n:n\geq 4\}|-1 \\&=|\{4, 5\}|-1\\&=1.
	\end{align*}
	Therefore by Theorem \ref{complsec4.01}, $uC_\phi$ is a Fredholm operator with $\text{ind}(uC_\phi)=3.$	
\end{example}

\begin{example}
	Let $u:\mathbb{N} \to \mathbb{C}$ be defined as $$ u(n)=\left\{\begin{array}{rcl} 0 & \mbox{if} & 1\leq n \leq 2 \\ 1+\frac{1}{n} & \mbox{if} & n>2
	\end{array}\right.$$
	and $\phi:\mathbb{N} \to \mathbb{N}$ be defined as
	$$ \phi(n)=\left\{\begin{array}{rcl} 1 & \mbox{if} & 1\leq n \leq 5 \\ n-2 & \mbox{if} & n>5
	\end{array}\right.$$
	
	Then $S(u)=\{n:n\geq 3\},~\mathbb{N}-\phi(S(u))=\{2, 3\}$ and  $$A=\{n\in \mathbb{N}:|\phi^{-1}(n)\cap S(u)|>1\}=\{1\}.$$
	Hence \begin{align*}
		\sum\limits_{n\in A}\left(|\phi^{-1}(n)\cap S(u)|-1\right)&=|\phi^{-1}(1)\cap \{n:n\geq 3\}|-1 \\&=|\{3, 4, 5\}|-1\\&=2.
	\end{align*}
	Therefore by Theorem \ref{complsec4.01}, $uC_\phi$ is a Fredholm operator with $\text{ind}(uC_\phi)=0.$
	
\end{example}

\begin{example}
	Let $u:\mathbb{N} \to \mathbb{C}$ be defined as $$ u(n)=\left\{\begin{array}{rcl} 0 & \mbox{if} & 1\leq n \leq 2 \\ 1+\frac{1}{n} & \mbox{if} & n>2
	\end{array}\right.$$
	and $\phi:\mathbb{N} \to \mathbb{N}$ be defined as
	$$ \phi(n)=\left\{\begin{array}{rcl} 1 & \mbox{if} & 1\leq n \leq 5 \\ n-3 & \mbox{if} & n>5
	\end{array}\right.$$
	Then $S(u)=\{n:n\geq 3\},~\mathbb{N}-\phi(S(u))=\{2\}$ and  $$A=\{n\in \mathbb{N}:|\phi^{-1}(n)\cap S(u)|>1\}=\{1\}.$$
	Hence \begin{align*}
		\sum\limits_{n\in A}\left(|\phi^{-1}(n)\cap S(u)|-1\right)&=|\phi^{-1}(1)\cap \{n:n\geq 3\}|-1 \\&=|\{3, 4, 5\}|-1\\&=2.
	\end{align*}
	Therefore by Theorem \ref{complsec4.01}, $uC_\phi$ is a Fredholm operator with $\text{ind}(uC_\phi)=-1.$		
\end{example}

\begin{center}
\textbf{Acknowledgments}
\end{center}
The author wishes to  thank his doctoral advisor Harish Chandra of the Department of Mathematics, Banaras Hindu University for helpful discussions and comments.Thanks are also due to C.S.I.R. New Delhi for providing financial assistance through NET-JRF with grant \bf{ 09/013(0891)/2019-EMR-I}.


\begin{thebibliography}{20}	
	
	\bibitem{Salamon} T. Bühler, Dietmar A. Salamon. \textit{Functional analysis. Vol. 191}, American Mathematical Soc., 2018.
	
	\bibitem{Cowen paper} C. C. Cowen, \textit{Linear fractional composition operators on $H^2$}, Integral equations and operator theory, 11.2 (1988), 151-160.
	
	\bibitem{Cowen} C. C. Cowen, B. Maccluer, \textit{Composition operators on spaces of analytic functions}, CRC Press. Boca Raton, 1995.
	
	\bibitem{Flytzans}  E. Flytzanis, L. Kanakis, \textit{Measure preserving composition operators}, Journal of functional analysis, {\bf{73.1}}, (1987), 113-121.
	
	\bibitem{GowersBernard} T. W. Gowers, B. Maurey, \text{The unconditional basic sequence problem}, Journal of the American Mathematical Society 6, no. 4, (1993), 851-874.
	
	\bibitem{Jbbar}M.R. Jabbarzadeh, \textit{Weighted composition operators between $L^p$-spaces}, Bull. Korean	Math. Soc. {\bf{42}}, (2005), 369–378.
	
	\bibitem{Jabb}M. R. Jabbarzadeh, E. Pourreza, \textit{A note on weighted composition operators on $L^p$-spaces}, Bull. Iranian Math. Soc. {\bf{29}} (2003), 47–54.
	
	\bibitem{Koopman} B. O. Koopman, \textit{Hamiltonian system and transformations in Hilbert space}, Proc. Nat. Acad. Sci., USA 17(1931), 315-318.
	
	\bibitem{Kothe} G. Köthe, \textit{Hebbare lokalkonvexe Räume}, Mathematische Annalen, {\bf{165(3)}}, (1966), 181-195.
	
	
	\bibitem{kumarwco} R. Kumar, \textit{Weighted composition operators between two $L^p$-spaces}, Mat. Vesnik {\bf{61}}, (2009), 111–118.
	
	\bibitem{Limparte} J. Lampertt, \textit{On the isometries of certain function-spaces}, Pacific J. Math. 8 (1958), 459466.
	
	\bibitem{Lindenstrauss67} J. Lindenstrauss,  \textit{On complemented subspaces of m}, Israel Journal of Mathematics 5, no. 3, (1967), 153-156.
	
	\bibitem{Lindenzafriri} J. Lindenstrauss, L. Tzafriri,  \textit{Classical Banach spaces I Sequence spaces}, Ergebnisse der Mathematik und ihrer Grenzgebiete, Vol. 92. Springer-Verlag, Berlin-New York, 1977.
	
	\bibitem{Murray} F. J. Murray,  \textit{On complementary manifolds and projections in spaces $L_ {p}$ and $l_ {p},$} Transactions of the American Mathematical Society {\bf{41}}, no.1(1937), 138-152.
	
	\bibitem{Nordgren} E. A. Nordgren,\text{ Composition operators}, Canad. J. Math., {\bf{20}} (1968), 442-449.
	
	\bibitem{Nordgren 2} E. A. Nordgren, \textit{Composition operators on Hilbert spaces, in Hilbert Space Operators} (Proc. Conf., Calif. State Univ., Long Beac Calif., 1977), pp. 3763, Lecture Notes in Math., 693, Springer, Berlin, (1978).
	
	\bibitem{Pelczynski} A. Pełczyński, \textit{Projections in certain Banach spaces}, Studia Mathematica 2, no. 19 (1960), 209-228.
	
	\bibitem{Phillips} R. S. Phillips, \textit{On linear transformations}, Transactions of the American Mathematical Society, {\bf{48(3)}}(1940), 516-541.
	
	\bibitem{Rolewicz}
	S. Rolewicz, \textit{On orbits of elements}, Studia Math. {\bf{32}} (1969), 17–22.
	
	\bibitem{Rosenthal} H.P. Rosenthal, \textit{On complemented and quasi-complemented subspaces of quotients of $C(S)$ for Stonian S.} Proceedings of the National Academy of Sciences, {\bf{60(4)}}, (1968), pp.1165-1169.
	
	
	\bibitem{W Rudin} W. Rudin, \textit{Functional analysis 2nd ed.}, International Series in Pure and Applied Mathematics, McGraw-Hill, Inc., New York (1991).
	
	\bibitem{Ryff} J. V. Ryff, \textit{Subordinate $H^p$ functions}, Duke J. Math. {\bf{33}}(1966)347-354.
		
	\bibitem{RKS2}
	R. K. Singh, J. S. Manhas,
	\textit{Composition operators on function spaces}, North-Holland, New York, 1993.
	
	\bibitem{AhmedR}
	A. R. Sourour, \textit{The isometries of $L_p (\Omega, X),$} Journal of Functional Analysis {\bf{30}}, no. 2 (1978): 276-285.
	
	\bibitem{Takagi}
	H. Takagi, \textit{Compact weighted composition operators on $L_,p$} Proceedings of the American Mathematical Society {\bf{116}}, no. 2 (1992) 505-511.
	
\end{thebibliography}
\end{document}